\documentclass[12pt]{article}

\usepackage{amsbsy,amsfonts,amsmath,amssymb,amsthm
,bbm,enumerate,euscript,graphicx}
\usepackage{color}

\usepackage{bm}

\usepackage{comment}

\usepackage{amsthm}

\theoremstyle{plain}
	\newtheorem{Theo}{Theorem}[section]
	\newtheorem{Prop}[Theo]{Proposition}
		    
	\newtheorem{Lemm}[Theo]{Lemma}

	\newtheorem{TheoPrinc}{Theorem}

\theoremstyle{definition}
	\newtheorem{Defi}[Theo]{Definition}

\theoremstyle{remark}

\def\emptyset{\varnothing}

\def\ZZd{\mathbb{Z}^d}
\def\RR{{\mathbb R}}    
\def\RRd{\mathbb{R}^d}

\DeclareMathOperator{\defin}{def} 
\DeclareMathOperator{\per}{per} 
\DeclareMathOperator{\sta}{sta} 

\def\F{\mathcal{F}} 
\def\Ln{\Lambda_n} 

\def\g{\omega} 
\def\gD{\omega_{\Delta}} 
\def\gDc{\omega_{\Delta^c}} 
\def\gDp{\omega'_{\Delta}} 
\def\guD{\omega_{\tau^{-1}_{u}(\Delta)}} 
\def\guDc{\omega_{\tau^{-1}_{u}(\Delta)^c}} 
\def\guDp{\omega'_{\tau^{-1}_{u}(\Delta)}} 

\def\fD{f_{\Delta}} 
\def\fDlk{f_{\Delta, k}^{l}}
\def\fDlK{f_{\Delta, K}^{l}}
\def\fDK{f_{\Delta, K}}

\def\ZD{Z_{\Delta}(\omega^{}_{\Delta^c})}
\def\ZDlk{Z^{l}_{\Delta, k}}

\def\ZDK{Z_{\Delta, K}}

\def\HD{H_{\Delta}}
\def\HDl{H_{\Delta}^l}
\def\HuD{H_{\tau^{-1}_{u}(\Delta)}}

\def\pizD{\pi_{\Delta}}

\def\pizLn{\pi_{\Lambda_n}}
\def\pizuD{\pi_{\tau^{-1}_{u}(\Delta)}}

\def\indK{\bm{1}_{\vert\gD\vert \leq K}^{}} 
\def\indu{\bm{1}_{\vert\guD\vert \leq K}^{}} 
\def\indpk{\bm{1}_{\vert\gDp\vert \leq k}^{}} 
\def\indpK{\bm{1}_{\vert\gDp\vert \leq K}^{}} 
\def\indpu{\bm{1}_{\vert\guDp\vert \leq K}^{}} 
\def\induD{\bm{1}_{\vert\tu(\g)_{\Delta}\vert \leq K}^{}} 

\def\Pbarn{\bar{P}_n}
\def\leb{\lambda^d} 
\def\CDl{C^l_{\Delta}}
\def\normf{\Vert f \Vert_{\infty}^{}} 
\def\tu{\tau_u} 
\def\t-u{\tau^{-1}_{u}} 

\begin{document}
\title{Existence of Gibbs point processes with stable infinite range interaction}

\author{
David Dereudre
$^{1}$ and Thibaut Vasseur $^{2}$\\
{\normalsize{\em }}
 }

\maketitle

 \footnotetext[1]{\;University of Lille, david.dereudre@univ-lille.fr }
 \footnotetext[2]{\;University of Lille, thibaut.vasseur@univ-lille.fr }

\begin{abstract}

We provide a new proof of the existence of Gibbs point processes with infinite range interactions, based on the compactness of entropy levels. Our main  existence theorem holds under two assumptions. The first one is the standard stability assumption, which means that the energy of any finite configuration is super-linear with respect to the number of points. The second assumption is the so-called  intensity regularity, which controls the long range of the interaction via the intensity of the process. This assumption is new and introduced here since it is well adapted to the entropy approach. As a corollary of our main result we improve the existence results by Ruelle for pairwise interactions \cite{Ruelle70} by relaxing the superstabilty assumption. Note that our setting is not reduced to pairwise interaction and can contain infinite range multi-body counterparts.

    \bigskip

\noindent {\it Key words: DLR equations, entropy bounds, superstable interaction.} 

\end{abstract}

\section{Introduction}

The Gibbs point processes constitute  a large class of point processes with interaction between the points. The interaction can be attractive or repulsive, depending on geometrical features, whereas the null interaction is associated  to the so-called Poisson point process. The existence of such processes in the infinite volume regime has a long history and is initially related to the existence of thermodynamic behaviours in statistical physics. Now the Gibbs point processes are used in several other applied sciences such as material science, astronomy, epidemiology, plant ecology, seismology, telecommunications, and others. Therefore their existence in the infinite volume regime is also relevant for spatial statistics considerations. In the present paper we give a new proof of the existence of such processes for a large class of infinite range interactions. 

The starting point of the theory is an energy function $H$ defined on the space of locally finite configurations in $\RRd$. In the following, $\omega$ denotes such a point configuration and $H(\omega)$ its energy. Then the finite volume Gibbs measure on a bounded set $\Lambda \subset \RRd$ is simply the probability measure
\begin{equation*}
P_\Lambda=\frac{1}{Z_\Lambda} e^{-H}\pi_\Lambda,    
\end{equation*}
where $\pi_\Lambda$ is the Poisson point process in $\Lambda$ with intensity one and $Z_\Lambda$ the normalization constant. The existence of $P_\Lambda$ is guaranteed by the stability condition recalled below. The existence of an infinite volume measure, corresponding to the case ``$\Lambda=\RRd$'', is not obvious and can not be achieved by the definition above. In fact, the general strategy is first to obtain a suitable thermodynamic limit for $(P_\Lambda)$  when $\Lambda$ tends to $\RRd$ and then derive a good description of the limiting point by the so-called DLR equations. The first general result in this direction is due to Ruelle in the seventies  \cite{Ruelle70}. The setting was the  pairwise interaction;
\begin{equation*}
    H(\omega)= \sum_{x\neq y\in\omega} \phi(|x-y|),
\end{equation*}
 where the potential $\phi$ is assumed {\bf Regular}, which means roughly that $\phi$ is summable at infinity (i.e. $\int_R^{+\infty} r^{d-1} |\phi(r)|dr<+\infty$ for $R>0$ large enough) and {\bf Superstable}: there exists a constant $A$ and for any bounded set $\Lambda \subset \RRd$ a constant $B_\Lambda >0$ such that for any finite configuration $\omega$ in $\Lambda$
 \begin{equation*}
  H(\omega) \ge A|\omega|+ B_\Lambda|\omega|^2,   
 \end{equation*}
where $|\omega|$ denotes the number of points in $\omega$.  Under both these assumptions, Ruelle proved the existence of at least one infinite Gibbs point processes. Similar results has been proved more recently in \cite{KKP04,KPR12} using functional analysis tools. In any case the superstability assumption  is required. As a corollary of our main result we improve these existence results in substituting the superstability assumption by the {\bf Stability} assumption: there exists a constant $A$ such that for any finite configuration $\omega$ 
\begin{equation*}
    H(\omega) \ge A|\omega|.
\end{equation*}
 
Let us note that the difference between stable and superstable potential is in fact weak since any stable potential becomes superstable if a pairwise continuous non-negative and non-null at origin potential is added. However there exists several pairwise stable potential which are not superstable. For instance any continuous, non-negative pairwise potential null at origin is stable without being superstable. Let us note also that several examples of stable (and non-superstable) energy functions have introduced recently in stochastic geometry and spatial statistics \cite{BVL,DDG12,MH08}. 

A multi-body interaction occurs when the energy function $H$ is decomposed using potentials on pairs, triplets, quadruplets, k-uplets of points (but not only on pairs). Our results do not use and do not depend on such decompositions. Therefore our existence result covers several multi-body interactions, including  infinite range cases (examples are given in Section \ref{Section_examples}). Let us mention that the finite range multi-body interaction have been treated in \cite{BP02} using the Dobrushin's criterium. 

Our main tool is the compactness of entropy level sets for the local convergence topology. This tool is particularly efficient for proving the tightness  of the sequence of finite volume Gibbs point processes. It has been used for the first time in \cite{GH96} and then a collection of papers followed \cite{dereudre2009existence,DDG12,DH15}. Before the present paper, the entropy strategy have been used only in the setting of finite range or random finite range interaction. As far as we know, it is the first time that it is applied in the setting of pure infinite range interaction. Therefore, our main contribution here was to developed a way to control the decay of the interaction adequately with the entropy approach. It is the reason why we introduce the {\bf Intensity Regular} assumption (see Definition \ref{definition_intReg}). It is called Intensity Regular because  the decay of the interaction is controlled via the intensity of the process. This choice is directly related to the entropy bounds which gives a uniform control of intensities for finite volume Gibbs processes within the thermodynamic limit. In the setting of pairwise interaction, our definition is similar to the Regular assumption by Ruelle.

As mentioned before, in the setting of pairwise interaction, the entropy strategy improves the existence results in relaxing the superstability assumption. However we lose Ruelle's estimates which ensures existence of moments of any order and some exponential and super-exponential moments. The entropy approach only provides moments of order one. Better estimates have to be obtained by different tools.

In the following Section \ref{Section_not}, we introduce the definitions and notations for Gibbs point processes. Our main existence Theorem is also given. Examples of energy functions are presented in Section \ref{Section_examples}. Finally, Section \ref{Section_Proof} is devoted to the proof of the theorem.
\newpage

\section{Notations and results}\label{Section_not}
The real $d$-dimensional space $\RRd$ is  equipped with the usual Euclidean distance $\Vert.\Vert$ and its associated Borel $\sigma$-algebra. Any set $\Lambda\subset \RRd$ is assumed measurable.
\subsection{Finite volume measure}

The space of configurations is the set of locally finite subsets of $\RRd$ 
$$
\Omega = \{\g\subset\RRd : \vert\omega\cap\Delta\vert < \infty,\; \forall \Delta \subset\RRd,\, \text{bounded} \},
$$ 
where $\vert \cdot \vert$ is the cardinal. We denote $\omega \cap \Delta$ by $\gD$, the union $\omega\cup\omega'$ of two configurations by $\omega\omega'$, the space of finite configurations by $\Omega_f$ and the space of configurations in  $\Lambda\subset \RRd$ by $\Omega_{\Lambda}$.

Our space is equipped with the sigma field $\mathcal{F}$ generated by the counting functions $N^{}_{\Delta} : \omega \mapsto \vert \gD \vert$ for all $\Delta \subset \RRd$ bounded. In our setting, a point process is simply a probability measure on $(\Omega, \mathcal{F})$. Note that, with this definition, we identify a point process with its distribution. We say that the process has a finite intensity if for all bounded  subset $\Delta$ the expectation $E_{P}[\vert \gD \vert]$ is finite. If we denote this expectation $\mu(\Delta)$ then $\mu$ is a sigma finite measure on $\RRd$ called the intensity measure. When $\mu = i(P) \leb$, where $\leb$ is the Lebesgue measure on $\RRd$ and $i(P) \geq 0$, we simply say that the point process has intensity $i(P)$. We also introduce 
\begin{equation*}
\xi(P) = \sup_{\substack{\Lambda \subset \RRd \\ 0 < \leb(\Lambda) < +\infty}} \frac{E^P[\vert \g_{\Lambda}^{}\vert]}{\leb(\Lambda)}
\end{equation*}
and we say that a probability measure $P$ has a bounded intensity if $\xi(P) < +\infty$. Obviously, if $P$ has a finite intensity $i(P)$ then $\xi(P) = i(P)$. 

A point process $P$ is stationary if for all $u\in\RRd$, $P = P\circ\t-u$, where $\tau_u$ is the translation of vector $u$. If a stationary point process has a finite intensity then its intensity measure is proportional to the Lebesgue measure and has the form $\mu = i(P) \leb$.

The most popular point processes are the Poisson point processes. We consider here only the homogeneous (or stationary) case where the intensity has the form  $\mu = \zeta \leb$. The process is denoted $\pi^{\zeta}$, or simply $\pi$ if $\zeta = 1$. Recall briefly that $\pi^{\zeta}$ is the only point process in $\RRd$ with intensity $\mu = i(P) \leb$ such that  any two disjoint regions of space are independent under $\pi^{\zeta}$. See the recent book  \cite{last_Penrose} on the subject.

 Let us now define the interaction between the points. We need to introduce an energy function.
\begin{Defi}\label{d.energie}
An energy function is a measurable function $H$ on the space of finite configurations $\Omega_f$ with values in $\RR \cup \{+\infty \}$ such that:
\begin{itemize}
\item[-] $H$ is \emph{non degenerate}: $H(\emptyset) < +\infty$,
\item[-] $H$ is \emph{hereditary}: for all $\omega \in \Omega_f$ and $x\in \omega$ then 
$$
H(\omega) < +\infty \Rightarrow H(\omega\setminus\{ x\}) < +\infty,
$$
\item[-] $H$ is \emph{stationary}: for all $\omega \in \Omega_f$ and $u\in\RRd$
$$ H(\tau_u\omega)=H(\omega).$$
\end{itemize}
\end{Defi}

A crucial assumption is the stability of the energy function.  
\begin{Defi}\label{d.stability} 
An energy function is said to be {\bf stable} if  there exists a constant $A \leq 0$ such that for all $\omega \in \Omega_f$ 
$$
H(\omega) \geq A \vert \g \vert.
$$
\end{Defi}
The assumption {\bf [Stable]}  is standard and have been treated deeply in the literature (see for instance Section 3.2 in \cite{ruelle1999statistical}).

Now we can define the Gibbs point processes in finite volume.
\begin{Defi}\label{d.mesurevolfini} Let $\Lambda$ be a bounded subset in $\RRd$. The Gibbs point process on $\Lambda$ for the stable energy function $H$ is the probability measure on $\Omega_{\Lambda}$ defined by 
$$
P_{\Lambda}(d\omega) = \frac{1}{Z_{\Lambda}}  e^{-H(\omega)} \pi_{\Lambda}(d\omega)
$$
with the  normalization constant $Z_{\Lambda} = \int e^{- H(\omega)} \pi_{\Lambda}(d\omega)$ called the partition function.
\end{Defi}

We can check with the properties of $H$ that $P_{\Lambda}$ is well defined (i.e. $0 < Z_{\Lambda} < +\infty$). In comparison with the standard formalism of Gibbs measures in statistical physics (see \cite{ruelle1999statistical} for instance), the activity and inverse temperature parameters are included in the function $H$ here.

\subsection{Infinite volume measure}

Let us turn now to the definition of Gibbs point processes in the infinite volume regime. We need to introduce the local energy which is given for a finite configuration $\g\in\Omega_f$ by $\HD(\g) = H(\g) - H(\gDc)$. It represents the contribution of energy coming from $\gD$ in $\g$ (the difference of energies with and without $\gD$).  We need to extend this definition for an infinite configuration $\g$. If $(\Delta_l)_{l\geq 0}$ is an increasing sequence of subsets in $\RRd$, we expect that the following limit exists "$\lim_{l\rightarrow +\infty} H^{l}_{\Delta}(\g)$" where 
$$H^{l}_{\Delta}(\g) =  H(\g^{}_{\Delta_l}) - H(\g^{}_{\Delta_l \setminus \Delta}).$$

 In particular, the difference $H^{l+1}_{\Delta}(\g) - H^l_{\Delta}(\g)$ should go to zero as $l$ goes to infinity. Our main assumption is a control of the expectation of this difference for point processes with bounded intensities.

\begin{Defi} \label{definition_intReg} An energy function is said {\bf intensity regular } if for all bounded subset $\Delta$ of $\RRd$, we can find an increasing sequence of subsets $(\Delta_l)_{l\geq 0}$ such that $\Delta \subset \Delta_0$ and 
\begin{equation*}
\left| H^{l+1}_{\Delta}(\g) - H^l_{\Delta}(\g) \right| \leq \vert \gD \vert G_{\Delta}^l(\gDc),
\end{equation*}
where $G^{l}_{\Delta}$ is a non negative function on $\Omega_\Delta^c$ such that for any probability measure $P$ 
\begin{equation*}
E^P\left[G_{\Delta}^l(\gDc) \right] \leq \alpha_l\; \psi(\xi(P)),
\end{equation*}
with $\psi$ an increasing function  and $(\alpha_l)_{l\geq 0}$ a sequence satisfying  $\displaystyle \sum_{l=0}^{+\infty} \alpha_l < +\infty$. 
\end{Defi}

From assumption {\bf [intensity regular]}, we have 
\begin{equation*}
\sum_{l=0}^{+\infty}\left| H_{\Delta}^{l+1}(\g) - H_{\Delta}^l(\g)  \right| \leq \vert \gD \vert \sum_{l=0}^{+\infty} G^l_{\Delta}(\gDc),
\end{equation*}
and we can control the expectation of the second part of the right side 
\begin{align*}
E^P\left[\sum_{l=0}^{+\infty} G^l_{\Delta}(\gDc) \right] \leq   \psi(\xi(P)) \sum_{l=0}^{+\infty} \alpha_l < + \infty.
\end{align*}
Then the local energy is correctly defined for $P$-almost every configurations $\g$ (with $P$ a probability measure with bounded intensity) by
\begin{equation*}
\HD(\g) = H_{\Delta}^{0}(\g) + \sum_{l=0}^{+\infty} \left[H_{\Delta}^{l+1}(\g) - H_{\Delta}^l(\g)  \right].
\end{equation*}
In particular, if $\g$ is a finite configuration in a bounded set $\Lambda$, it comes that 
\begin{equation*}
\HD(\g) = H(\g_{\Lambda}) - H(\g_{\Lambda\setminus\Delta}).
\end{equation*}
In addition, if we introduce the function $C^l_{\Delta}$ on $\Omega_\Delta^c$,by 
\begin{equation*}
C^l_{\Delta}(\gDc) = \sum_{j=l}^{+\infty} G^l_{\Delta}(\gDc),
\end{equation*}
then for all configuration $\omega \in \Omega$ we have the folowing approximation result 
\begin{equation}\label{approx.local.energy}
\left| \HD(\g) - H_{\Delta}^l(\g) \right| \leq \vert \gD \vert C_{\Delta}^l(\gDc)
\end{equation}
with $\displaystyle E^P[C_{\Delta}^l(\gDc)] \leq \psi(\xi(P)) \sum_{j=l}^{\infty}\alpha_j$. 

We are now able to give the definition of infinite volume Gibbs point processes.

\begin{Defi}\label{d.mesurevolinfini}
A probability measure $P$ on $\Omega$ with bounded intensity is  a Gibbs point process for an energy function $H$, satisfying assumptions {\bf [Stable]} and {\bf [Intensity Regular]}, if for all bounded subset $\Delta$ and all bounded measurable function $f$ we have
\begin{equation}\label{DLRequations}
\int f(\g) P(d\g) 
=
\int \int f(\gDp\gDc) \frac{1}{\ZD} e^{- \HD(\gDp\gDc)} \pi_{\Delta}(d\gDp) P(d\g),
\end{equation}
with the normalization constant $\ZD = \int e^{- \HD(\gDp\gDc)} \pi_{\Delta}(d\gDp)$.
\end{Defi}
The equations \eqref{DLRequations} for all $\Delta$ and $f$ are called DLR for Dobrushin-Lanford-Ruelle. To be correctly defined we need to check that $0 < \ZD < +\infty$ for $\g$ sampled from $P$. A lower bound can easily be obtained with
\begin{equation*}
\ZD \geq e^{-\HD(\gDc)}\pi_{\Delta}(\emptyset) = e^{-\leb(\Delta)} > 0.
\end{equation*}
To have an upper bound, we use  assumptions {\bf [Stable]} and {\bf [Intensity Regular]}. From inequality \eqref{approx.local.energy}
\begin{align*}
\HD(\g) & \geq H_{\Delta}^{0}(\g) - \vert \gD \vert C_{\Delta}^{0}(\gDc) \\
& = H(\g_{\Delta_0}) - H(\g_{\Delta_0 \setminus \Delta}) - \vert \gD \vert C_{\Delta}^{0}(\gDc) \\
& \geq (A - C_{\Delta}^{0}(\gDc))  \vert  \gD \vert + A \vert \g_{\Delta_0 \setminus \Delta}\vert - H(\g_{\Delta_0 \setminus \Delta}).
\end{align*}
Then we obtain
\begin{align*}
&\ZD \\& \leq \exp\left(-A \vert \g_{\Delta_0 \setminus \Delta}\vert + H(\g_{\Delta_0 \setminus \Delta})\right)\int e^{ (C_{\Delta}^{0}(\gDc) - A) \vert \gDp \vert} \pi_{\Delta}(d\gDp) \\
& = \exp\left(- A \vert \g_{\Delta_0 \setminus \Delta}\vert + H(\g_{\Delta_0 \setminus \Delta})\right) \sum_{n=0}^{+\infty} \frac{\left(\leb(\Delta)e^{ (C_{\Delta}^{0}(\gDc) - A))}\right)^n}{n!} e^{-\leb(\Delta)} \\
& = \exp\left(- A \vert \g_{\Delta_0 \setminus \Delta}\vert + H(\g_{\Delta_0 \setminus \Delta}) +\leb(\Delta)\left(e^{ (C_{\Delta}^{0}(\gDc) - A))} - 1\right)\right) \\
& < +\infty.
\end{align*}

Our main result is the following theorem which is proved in Section \ref{Section_Proof}.

\begin{TheoPrinc}\label{t.DLR} 
For any energy function $H$ satisfying assumptions {\bf [Stable]} and {\bf [Intensity Regular]}, there exists at least one stationary Gibbs point process with finite intensity.
\end{TheoPrinc}

\section{Examples}\label{Section_examples}

Let us give examples of energy functions satisfying assumptions {\bf [Stable]} and {\bf [Intensity Regular]} of Theorem  \ref{t.DLR}. Since the assumption {\bf [Stable]} is studied deeply in the literature, we focus mainly on interesting examples satisfying assumption {\bf [Intensity Regular]}.

\paragraph{Finite range interaction.} An energy function $H$ has a finite range if there exists $R>0$ such that for all finite configurations and subset $\Delta$ bounded $\omega \in \Omega_f$, $H_{\Delta}(\omega) = H_{\Delta}\left(\omega_{\Delta\oplus B(0,R)}\right)$. It is easy to see that a finite range energy function verifies assumption  {\bf [Intensity Regular]}, with any increasing sequence of subsets $(\Delta_l)_{l\geq 0}$ such that $\Delta_0 = \Delta\oplus B(0,R)$. Therefore in this setting of finite range interaction, only the assumption {\bf [Stable]} is required to ensure the existence of Gibbs point processes. This result has been proved previously in \cite{DDG12} (See also \cite{dereudre2017introduction} for a simpler and pedagogical proof).

\paragraph{Pairwise interaction.} An energy function  $H$ is pairwise if there exists a symmetric function $\Phi : \RRd \mapsto \RR \cup\{\infty\}$, called a potential, such that  
\begin{equation*}
H(\omega)= \sum_{\{x,y\} \subset \omega} \Phi(x - y).
\end{equation*}
We assume that the potential is isotropic (i.e. $\Phi(x - y) = \phi(\Vert x - y\Vert)$ with $\phi : \RR_+ \mapsto \RR \cup\{\infty\}$). We do not assume the finite range property and so the support of $\Phi$ can be unbounded.

If there exists an integer $L_{\phi}\ge 0$ such that
\begin{equation}\label{h.longue.range}
    \displaystyle \sum_{l=L_{\phi}}^{\infty} l^{d-1} \sup_{r\in [l,l+1]} \vert \phi(r)\vert < \infty
\end{equation}
then the energy function $H$ satisfies the assumption {\bf [Intensity Regular]}. Indeed, let $\Delta$ be a bounded subset of $\RRd$, if we define $\Delta_l = \Delta \oplus B(0,L_{\phi}+l)$, we have 
\begin{equation*}
    H^l_{\Delta}(\omega) = \sum_{\substack{\{x,y\}\subset\omega\cap\Delta_l\\ \{x,y\}\cap\Delta \neq \emptyset}} \phi(\Vert x- y \Vert)
\end{equation*}
and then 
\begin{equation*}
    H^{l+1}_{\Delta}(\omega) - H^l_{\Delta}(\omega) = \sum_{x\in\Delta} \sum_{y\in\Delta_{l+1}\setminus\Delta_l} \phi(\Vert x-y \Vert).
\end{equation*}
By definition, if $x\in\Delta$ and $y\in\Delta_{l+1}\setminus\Delta_l$ then $\Vert x - y \Vert \in [l,l+1+\delta]$ with $\delta = \text{diam}(\Delta)$. Denoting
\begin{equation*}
    G^l_{\Delta}(\gDc) = \vert \omega_{\Delta_{l+1}\setminus\Delta_l} \vert \sup_{r\in[l,l+1+\delta]}\vert \phi(r)\vert,
\end{equation*}
we have 
$\vert H^{l+1}_{\Delta}(\omega) - H^l_{\Delta}(\omega) \vert \leq \vert \gD \vert G^l_{\Delta}(\gDc)$
with $E_P\left[G^l_{\Delta}(\gDc)\right] \leq \alpha_l \xi(P)$
where
\begin{equation*}\label{h.longue.range}
    \alpha_l = c_d\, l^{d-1} \sup_{r\in[l,l+1+\delta]} \vert \phi(r)\vert 
\end{equation*}
is such that $\sum \alpha_l < \infty$.

Note that any combination of such a pairwise energy function and any finite range energy  function also verifies assumption {\bf [Intensity Regular]}  (with $\Delta_l = \Delta \oplus B(0,R\lor L_{\phi} + l)$).

The assumption {\bf [Stable]} is more delicate and has long been investigated, we refer to \cite{ruelle1999statistical} for several results. 

\paragraph{Cloud interaction}

In this last example, we provide an energy function which is  infinite range and not reducible, at any scale, to a pairwise interaction. It is a multibody  interaction between a germ-grain interaction (see for instance the Quermass model \cite{dereudre2009existence} or the Widom-Rowlinson interaction \cite{WR}) and a pairwise interaction. We call it cloud interaction because each point of the configuration is diluted in a cloud around itself and the pairwise interaction is integrated on this cloud.  Precisely for any finite configuration $\omega$

\begin{equation*}
    H(\omega) = \sum_{x\in\omega}\int_{L^R(\omega)} \phi(\Vert x - y \Vert) dy
\end{equation*}
where  $\displaystyle L^R(\omega) = \bigcup_{x\in\omega} B(x,R)$ is the cloud produced by the configuration $\omega$ ($R>0$ is a fixed parameter). This energy function can be viewed as an approximation of the pairwise interaction introduced above. Indeed $H(\omega)/R^d$ tends to the pairwise interaction function (times a multiplicative constant) when $R$ goes to zero.

We suppose that the potential $\phi$ satisfies $\int r^{d-1} \vert\phi(r)\vert dr < +\infty$. But we need a slightly stronger assumption such as (\ref{h.longue.range}). So, in order to simplify, we assume that the potential is monotonic at large distances.

The energy function satisfies clearly assumption {\bf [Stable]} since
\begin{equation*}
|H(\omega)| \le  |\omega| \int r^{d-1} \vert\phi(r)\vert dr.
\end{equation*}
Note that $H$ is not superstable.

Involving assumption {\bf [Intensity Regular]}, the local energy is given by
\begin{align*}
    H^l_{\Delta}(\omega) & = H(\omega_{\Delta_l}) - H(\omega_{\Delta_l\setminus\Delta}) \\
    & = \sum_{x\in\omega_{\Delta_l}}\int_{L^R(\omega_{\Delta_l})} \phi(\Vert x - y \Vert) dy - \sum_{x\in\omega_{\Delta_l\setminus\Delta}}\int_{L^R(\omega_{\Delta_l\setminus\Delta})} \phi(\Vert x - y \Vert) dy \\
    & = \sum_{x\in\omega_{\Delta_l\setminus\Delta}}\int_{L^R(\omega_{\Delta_l})\setminus L^R(\omega_{\Delta_l\setminus\Delta})} \phi(\Vert x - y \Vert) d y + \sum_{x\in\gDc} \int_{L^R(\omega_{\Delta_l})} \phi(\Vert x - y \Vert) d y,
\end{align*}
For the following, we choose $\Delta_0=\Delta \oplus B(0,2R)$, which implies that 
\begin{equation*}
    L^R(\omega_{\Delta_l})\setminus L^R(\omega_{\Delta_l\setminus\Delta})
    = L^R(\omega_{\Delta_0})\setminus L^R(\omega_{\Delta_0\setminus\Delta}) 
    \overset{\defin}{=} L^R_{\Delta}(\omega).
\end{equation*}
Using this notation we have 
\begin{equation*}
    H^l_{\Delta}(\omega)  = \sum_{x\in\omega_{\Delta_l\setminus\Delta}}\int_{L^R_{\Delta}(\omega)} \phi(\Vert x - y \Vert) d y + \sum_{x\in\gDc} \int_{L^R(\omega_{\Delta_l})} \phi(\Vert x - y \Vert) d y.
\end{equation*}
The first term corresponds to the interaction of the points outside of $\Delta$ with the cloud created by the points in $\Delta$ and the second corresponds to the interaction of the points in $\Delta$ with the full cloud.
We can compute the cost of adding a shell 
\begin{multline*}
    H^{l+1}_{\Delta}(\omega) - H^l_{\Delta}(\omega) = \sum_{x\in\omega_{\Delta_{l+1}}\setminus\omega_{\Delta_l}} \int_{L^R_{\Delta}(\omega)} \phi(\Vert x-y\Vert) dy \\ + \sum_{x\in\gDc}\int_{L^R(\omega_{\Delta_{l+1}})\setminus L^R(\omega_{\Delta_l})} \phi(\Vert x -y\Vert)dy.
\end{multline*}
If we choose $\Delta_l = \Delta_0 \oplus B(0,l) = \Delta \oplus B(0,2R+l)$, for $x\in\omega_{\Delta_{l+1}}\setminus\omega_{\Delta_l}$ and $y \in L^R_{\Delta}(\omega)$ or for $x\in\gDc$ and $y\in L^R(\omega_{\Delta_{l+1}})\setminus L^R(\omega_{\Delta_l})$ we have $\Vert x - y \Vert \in I_l$ where $I_l = [l+R,l+1+3R+\text{diam}(\Delta)]$, then we obtain the upper bound 
\begin{equation*}
    \vert H^{l+1}_{\Delta}(\omega) - H^l_{\Delta}(\omega)   \vert \leq \left(\vert \omega_{\Delta_{l+1}\setminus\Delta_{l}} \vert \leb(L^R_{\Delta}(\omega)) + \vert\gDc\vert \leb(L^R(\omega_{\Delta_{l+1}\setminus\Delta_l})) \right)\sup_{r\in I_l } \vert \phi(r) \vert.
\end{equation*}
The energy satisfies assumption {\bf [Intensity Regular]} with 
\begin{equation*}
    G_{\Delta}^l(\gDc) = \left(\vert \omega_{\Delta_{l+1}\setminus\Delta_{l}} \vert \leb(\Delta\oplus B(0,R)) +  \leb(\Delta_{l+1}\setminus\Delta_l\oplus B(0,R)) \right)\sup_{r\in I_l } \vert \phi(r) \vert.
\end{equation*}

\section{Proof of the theorem}\label{Section_Proof}

\subsection{Construction of an infinite volume measure}

The first step of the proof is to build an accumulation point of a sequence of finite volume Gibbs measures. Using entropy bounds and the stability of the energy, we prove the existence of such an accumulation point for the local convergence topology. This strategy and its tools have been used several time in the literature (\cite{dereudre2009existence}, \cite{dereudre2017introduction}, \cite{DDG12}, \cite{DH15}) and we recall here only the main ideas.

For $n$ a positive integer, we denote $\Ln= ]-n,n]^d$. We consider the sequence of Gibbs measures in finite volume given by 
$$
P_n(d\g) = P_{\Ln}(d\g) = \frac{1}{Z_n} e^{- H(\g)} \pizLn(d\g),
$$
with the normalization constant $Z_n = \int e^{- H(\g)} \pizLn(d\g)$. Since our tension tool will be defined for stationary measures, we need to modify $(P_n)_{n\geq1}$. 
We defined the periodized version $P_n^{\per}$ by the probability measure $\underset{u\in\ZZd}{\bigotimes} P_n \circ \tau^{-1}_{2nu}$, and the stationnarized version by 
\begin{equation*}
P^{\sta}_n = \frac{1}{\leb(\Ln)}\int_{\Ln} P_n^{\per} \circ \t-u d u. 
\end{equation*}

\begin{Defi}
A function $f$ is said local if there exists a bounded set $\Delta$ such that $f$ is $\F_{\Delta}$ measurable (ie for all configuration $\g$ in $\Omega$, we have $f(\g) = f(\gD)$). A sequence of measures $(\mu_n)$ converges to $\mu$ for the local convergence topology if for all bounded local functions $f$ 
\begin{equation*}
\int f d\mu_n \underset{n\rightarrow +\infty}{\rightarrow} \int f d\mu.
\end{equation*}
\end{Defi}

Given two probabilities measures  $\mu$ and $\nu$ on $\Omega$, we recall that the relative entropy of $\mu$ with respect to $\nu$ on $\Lambda$ is defined as 
$$
I_{\Lambda}(\mu | \nu) = \left\{
    \begin{array}{ll}
        \int \log f d\mu_{\lambda} & \mbox{ if } \mu_{\Lambda}  \ll \nu_{\Lambda} \mbox{ and } f = \frac{d\mu_{\Lambda}}{d\nu_{\Lambda}}   \\
        +\infty & \mbox{otherwise.}
    \end{array}
\right.
$$ 
\begin{Defi}\label{d.specific.entropy}
Let $\mu$ be a stationary probability measure with finite intensity on $\Omega$. For $\zeta > 0$ the specific entropy of $\mu$ with respect to $\pi^{\zeta}$ is defined by
\begin{equation}\label{formule.specific.entropy}
I_{\zeta}(\mu | \pi^{\zeta}) =  \lim_{n\rightarrow +\infty} \frac{I_{\Ln}(\mu\, | \pi^{\zeta})}{\leb(\Ln)} 
=
\sup_{\substack{\Lambda \subset \RRd \\ 0 < \leb(\Lambda) < +\infty}} \frac{I_\Lambda(\mu|\pi^\zeta)}{\leb(\Lambda)}.
\end{equation}
\end{Defi}
For details we refer to Chapter 15 of \cite{georgii2011gibbs}. The next result, stated in \cite{georgii1993large}, is our tension tool.

%
%
%

\begin{Prop}\label{p.tension}
For any $\zeta > 0$ and $c > 0$, the set of probability measures 
$$
\{ \mu \mbox{ stationary with finite intensity, } I_{\zeta}(\mu) \leq c\}
$$
is compact and sequentially compact for the local convergence topology. 
\end{Prop}

In order to apply this proposition in our case, we need to compute the specific entropy of the probability measure $P^{\sta}_n$. Using the affine property of the specific entropy, it is well known that  

$$I_{\zeta}(P^{\sta}_n)= \frac{1}{\leb(\Ln)} I_{\Ln}(P_n | \pi^{\zeta}).$$

%
What remains is to compute the relative entropy of the Gibbs measure $P_n$ with respect to the Poisson point process $\pi^{\zeta}$ on $\Ln$;
\begin{align*}
I_{\Ln}(P_n | \pi^{\zeta}) 
& = \int \log\left(\frac{dP_n}{d\pi^{\zeta}}\right) dP_n \\
& = \int \left[\log\left(\frac{dP_n}{d\pi^z}\right) + \log\left(\frac{d\pi}{d\pi^{\zeta}}\right)\right] dP_n\\
& = \int \left[-\log(Z_n) -  H(\g) + \log\left(e^{(\zeta - 1)\leb(\Ln)}\left(\frac{1}{\zeta}\right)^{\vert \g \vert}\right)\right] P_n(d\g) \\
&= (\zeta - 1) \leb(\Ln) - \log(Z_n) + \int\left[ - H(\g) - \log(\zeta) \vert \g \vert \right] P_n(d\g).
\end{align*}
The normalization constant can easily be bounded from below
\begin{equation*}
Z_n = \int e^{- H} d\pi \geq e^{- H(\emptyset)} e^{-\leb(\Ln)}.
\end{equation*}
Then, using the stability of $H$, we have 
\begin{equation*}
I_{\Ln}(P_n | \pi^{\zeta})  \leq \zeta \leb(\Ln) +  H(\emptyset) + \int (- A - \log(\zeta) )\vert \g \vert P_n(d\g). 
\end{equation*}
If $\zeta$ is such that $- A - \log(\zeta) \leq 0$, we have  $I_{\zeta}(P^{\sta}_n) \leq \zeta +  H(\emptyset)$. According to Proposition \ref{p.tension} we can exhibit a sub-sequence of $(P_n^{\sta})_{n\geq 1}$ which converges to a stationary measure $P$ with finite intensity. To simplify the notations, we can suppose that we have changed the indexation of the sequence $(\Ln)_{n\geq 1}$ such that $(P^{\sta}_n)_{n\geq 1}$ converges locally to $P$.

We can prove that $P$ is also an accumulation point of the sequence 
\begin{equation*}
\Pbarn = \frac{1}{\leb(\Ln)} \int_{\Ln} P_n \circ \t-u du.
\end{equation*}
See Lemma 3.5 \cite{dereudre2009existence} for details.

%
%

Let us finish this section by giving the crucial property of uniform control of intensities for the sequence $(\bar P_n)$
\begin{Lemm}\label{l.bounded.intensity}
We can find $\xi \geq i(P)$ such that for all integer $n\geq 1$, $\xi(\Pbarn) \leq \xi$. 
\end{Lemm}
\begin{proof}
We use the entropic inequality $\mu(g) \leq I(\mu|\nu) + \log(\nu(e^g))$ to obtain
\begin{equation*}
E_{P^{\sta}_n}[\vert \omega_{\Ln} \vert] \leq I_{\Ln}(P^{\sta}_n | \pi^{\zeta}) + \log  \left(E_{\pi^{\zeta}_{\Ln}}\left[e^{\vert \omega_{\Ln} \vert}\right]\right).
\end{equation*}
Using the expression of the specific entropy as a supremum (\ref{formule.specific.entropy}), we have
\begin{equation*}
I_{\Ln}(P^{\sta}_n | \pi^{\zeta}) \leq I_{\zeta}(P^{\sta}_n) \leb(\Ln) \leq (\zeta +  H(\emptyset))\leb(\Ln).
\end{equation*}
Under $\pi^{\zeta}_{\Ln}$, the random variable $\vert \omega_{\Ln} \vert$ follows a Poisson law of parameter  $\zeta \leb(\Ln)$ and so
$$
E_{\pi^{\zeta}_{\Ln}}\left[e^{\vert \omega_{\Ln}\vert} \right] = e^{-\zeta \leb(\Ln)}\sum_{p=0}^{\infty}  \frac{(\zeta\leb(\Ln))^p}{p!}e^p = \exp(\zeta \leb(\Ln)(e-1)).
$$
Then we obtain 
$$
i(P^{\sta}_n) \leb(\Ln) = E_{P^{\sta}_n}[\vert \omega_{\Ln} \vert] \leq (\zeta e +  H(\emptyset))\leb(\Ln),
$$
and since $\xi(\bar{P}_n) \leq i(P^{\sta}_n)$, we deduce the lemma.
\end{proof}
\subsection{The DLR equation}

We prove in this section  that the accumulation point $P$ satisfies the DLR equations stated in Definition \ref{d.mesurevolinfini}. By a standard class monotone argument  we can replace the class of bounded measurable functions by the class of bounded local functions. Let $f$ be a bounded local function, $\Delta$ be a bounded measurable subset of $\RRd$, we show $\int f dP = \int \fD dP$ where 
\begin{equation*}
\fD(\omega) =  \int f(\gDp\gDc) \frac{1}{\ZD} e^{- \HD(\gDp\gDc)} \pi_{\Delta}(d\gDp).
\end{equation*}
We fix $\epsilon > 0$.

\paragraph{Step 1.} There exists $K' > 0$ such that for all $k \geq K'$, we have $P(\vert\gD\vert > k) \leq \varepsilon$, 
this implies that for all $k \geq K'$,

\begin{equation}\label{approx1}
\left| \int \fD(\g)P(d\g) - \int \fD(\g)\bm{1}_{\vert\gD\vert\leq k} P(d\g) \right| \leq \normf\varepsilon.
\end{equation}
The introduction of this indicator function will be usefull in step 4.
\paragraph{Step 2} We approach $\fD$ by $\fDlk$ which corresponds to the approximation of the local energy $\HD$ by $\HDl$ and a restriction to configurations having less than $k$ points in $\Delta$, which means

\begin{equation*}
\fDlk(\g) = \frac{1}{\ZDlk(\gDc)}\int f(\gDp\gDc)e^{-\HDl(\gDp\gDc)}\indpk \pizD(d\gDp),
\end{equation*}
with the normalization constant 
$\ZDlk(\gDc) = \int e^{-\HDl(\gDp\gDc)}\indpk \pizD(d\gDp)$.
We prove that we can find $K \geq K'$ and $l$ (depending on $K$) such that :
\begin{equation}\label{approx2}
\left|\int \fD(\g)\indK P(d\g) - \int \fDlK(\g) \indK P(d\g)  \right| \leq 6\normf\varepsilon.
\end{equation}
We must estimate the approximation error 
\begin{align*}
& \fD(\g) - \fDlk(\g) \\
& = \frac{1}{\ZD}\int f(\gDp\gDc) \left(e^{- \HD(\gDp\gDc)} - e^{-\HDl(\gDp\gDc)}\indpk\right) \pizD(d\gDp) \\
& \quad + \left(\frac{1}{\ZD}-\frac{1}{\ZDlk(\gDc)}\right)\int f(\gDp\gDc) e^{- \HDl(\gDp\gDc)}\indpk \pizD(d\gDp).
\end{align*}
Since the difference between the normalization constants is 
\begin{equation*}
\ZDlk(\gDc)-\ZD = \int \left(e^{-\HDl(\gDp\gDc)}\indpk e^{- \HD(\gDp\gDc)}\right) \pizD(d\gDp),
\end{equation*}
we obtain the upper-bound 
\begin{align}\label{e.majoration}
\vert \fD(\g) - \fDlk(\g)\vert \leq {} & \frac{2\normf}{\ZD}\int \left\vert e^{- \HD(\gDp\gDc)} - e^{-\HDl(\gDp\gDc)}  \indpk \right\vert\pizD(d\gDp) \nonumber\\
= {}  & \frac{2\normf}{\ZD}\int \left\vert e^{- \HD(\gDp\gDc)} - e^{-\HDl(\gDp\gDc)} \right\vert \indpk \pizD(d\gDp) \nonumber\\
 & + \frac{2\normf}{\ZD}\int e^{- \HD(\gDp\gDc)} \bm{1}_{\vert\gDp\vert > k}^{} \pizD(d\gDp).
\end{align}
By the dominated convergence theorem, we can find  $K \geq K'$ such that 
\begin{equation*}
\int\frac{1}{\ZD}\int e^{- \HD(\gDp\gDc)} \bm{1}_{\vert\gDp\vert > K}^{} \pizD(d\gDp)P(d\g)
\leq \varepsilon.
\end{equation*}
Once $K$ is chosen, using the inequality $\vert e^b- e^a \vert \leq \vert b -a \vert e^{\vert b - a\vert + a} $ and the approximation (\ref{approx.local.energy}) we obtain the upper-bound
\begin{equation}\label{e.majorationexp}
\left\vert e^{-\HDl(\gDp\gDc)} - e^{- \HD(\gDp\gDc)} \right\vert \indpK
\leq
 K \CDl(\gDc) e^{ K \CDl(\gDc)- \HD(\gDp\gDc)}.
\end{equation}
As $E^P[\CDl(\gDc)]$ goes to zero when $l$ goes to infinity, with Markov's inequality we can choose $l$ (depending to $K$) such that 
\begin{equation*}
P\left( K \CDl(\gDc) e^{ K \CDl(\gDc)} > \varepsilon\right) \leq \varepsilon.
\end{equation*}
According to Lemma \ref{l.bounded.intensity}, the point processes ($\Pbarn)_{n\geq1}$ have uniformly bounded intensities, so $l$ could be such that for all $n$
\begin{equation}\label{probaPn}
\Pbarn\left(K \CDl(\gDc) e^{ K \CDl(\gDc)} > \varepsilon\right) \leq \varepsilon,
\end{equation}
which will be useful later.
With our choice of $K$ and $l$ we have finally the approximation (\ref{approx2}).

\paragraph{Step 3.} For $n$ large enough (depending on $K$ and $l$) we have
\begin{equation}\label{approx3}
\left|\int \fDlK(\g) \bm{1}_{\vert\gD\vert \leq K} P(d\g) - \int \fDlK(\g) \bm{1}_{\vert\gD\vert \leq K} \Pbarn(d\g) \right| \leq \normf\varepsilon.
\end{equation}
It is simply a consequence of the local convergence of the sequence $(\Pbarn)_{n\geq 1}$ to $P$.

\paragraph{Step 4.}For all $n \geq 1$ we show the approximation 
\begin{equation}\label{approx4}
\left| \int \fDlK(\g) \bm{1}_{\vert\gD\vert \leq K} \Pbarn(d\g) - \int \fDK(\g) \bm{1}_{\vert\gD\vert \leq K} \Pbarn(d\g) \right| \leq 4 \normf \varepsilon,
\end{equation}
where  
\begin{equation*}
\fDK(\g) = \frac{1}{\ZDK(\gDc)}\int f(\gDp\gDc)e^{-\HD(\gDp\gDc)}\indpK \pizD(d\gDp),
\end{equation*}
with the normalization constant 
$\ZDK(\gDc) = \int e^{-\HD(\gDp\gDc)}\indpK \pizD(d\gDp)$.

Similarly to the upper-bounds (\ref{e.majoration}) and (\ref{e.majorationexp}) we obtain 
\begin{align*}
\vert \fDlK(\g) - \fDK(\g) \vert \leq {} & \frac{2\normf}{\ZDK(\gDc)}\int \vert e^{-\HD(\gDp\gDc)} - e^{-\HDl(\gDp\gDc)}\vert \bm{1}_{\vert\gDp\vert \leq K} \pizD(d\gDp) \\
\leq {} & 2\normf e^{ K \CDl(\gDc)} K \CDl(\gDc).
\end{align*}
From our previous choice of $K$ and $l$ in estimate (\ref{probaPn}),
we obtain the approximation (\ref{approx4}).

\paragraph{Step 5.} We use the DLR equations for finite volume Gibbs processes to prove that
\begin{equation}\label{approx5}
\left| \int \fDK(\g) \indK \Pbarn(d\g) - \int f(\g) \indK \Pbarn(d\g) \right|
\leq 2 \normf \varepsilon.
\end{equation}
Let us introduce $\Ln^* = \{ u \in \Ln : \t-u(\Delta) \subset \Ln \}$. Note that if $\Delta \subset \Lambda_k$ and $n \geq k$ then $\Lambda_{n-k} \subset \Ln^*$ and $(n-k)^d / n^d \leq \leb(\Ln^*) / \leb(\Ln) \leq 1$. We choose $n$ large enough such that $\leb(\Ln^*) / \leb(\Ln) \geq 1 - \varepsilon$, and if we denote 
\begin{equation*}
\Pbarn^* = \frac{1}{\leb(\Ln)}\int_{\Ln^*} P_n \circ \t-u,
\end{equation*} 
we have the approximation 
\begin{equation*}
\left| \int \fDK(\g) \bm{1}_{\vert\gD\vert \leq K} \Pbarn(d\g)  - \int \fDK(\g) \bm{1}_{\vert\gD\vert \leq K} \Pbarn^*(d\g)\right| \leq \normf \varepsilon.
\end{equation*}
Let us detail the term 
\begin{align*}
& \int \fDK(\g) \indK \Pbarn^*(d\g) \\
& = \frac{1}{\leb(\Ln)} \int_{\Ln^*} \int \fDK(\tu(\g))\induD P_n(d\g) du  \\
& = \frac{1}{\leb(\Ln)} \int_{\Ln^*} \iint f(\gDp\tu(\g)_{\Delta^c}) \frac{1}{\ZDK(\tu(\g)_{\Delta^c}^{})} e^{-\HD(\gDp \tu(\g)_{\Delta^c})}  \indpK\pizD(d\gDp) \\ & \hspace*{9cm} \induD P_n(d\g) du. &
\end{align*}
For $u\in\Ln^*$, using  the fact that $\tu(\g)_{\Delta^c} = \tu\big(\guDc\big) $, that $\pizD$ has the same law than $\pizuD \circ \t-u$ and that $\HD(\tu(\g)) =H_{\tau^{-1}_u(\Delta)}(\g)$, we have 
\begin{align*}
 \ZDK(\tu(\g)_{\Delta^c}^{})
& = \int e^{-\HD\big(\gDp\tu\big(\guDc\big)\big)}\bm{1}_{\vert\gDp\vert \leq K} \pizD(d\gDp) \\
& = \int e^{-\HD\big(\tu\big(\guDp\guDc\big)\big)}\bm{1}_{\big|\guDp\big|  \leq K} \pizuD\big(d\guDp\big)\\
& = \int e^{-\HuD\big(\guDp\guDc\big)}\bm{1}_{\big|\guDp\big| \leq K} \pizuD\big(d\guDp\big)\\
& = Z_{\t-u(\Delta), K}\big(\guDc\big).
\end{align*}
Then, by a similar calculation, we find 
\begin{multline*}
\int \fDK(\tu(\g))\induD P_n(d\g)  \\ =
\iint f\big(\tu\big(\guDp\guDc\big)\big) \frac{1}{Z_{\t-u(\Delta), K}\big(\guDc\big)} e^{- \HuD\big(\guDp\guDc\big)} \\ \indpu \pizuD(d\guDp) \indu P_n(d\g).
\end{multline*}
But we can write the measure in finite volume as 
\begin{multline*}
P_n(d\g) = \frac{1}{Z_n} e^{- \HuD\big(\guD\guDc\big)}e^{- H\big(\guDc\big)} \\ \pi^z_{\Ln\setminus\t-u(\Delta)}\big(d\guDc\big)  \pizuD\big(d\guD\big),
\end{multline*}
and integration with respect of the measure $\pizuD$ will give the normalization constant (thanks to the indicator function introduce in step 1). After simplification we have for the translated $P_n \circ \t-u$ with $u\in\Lambda^*_n$ a finite volume DLR equation  
\begin{align*}
&\int \fDK(\tu(\g))\induD P_n(d\g) \\
& = \frac{1}{Z_n} \iint f\big(\tu\big(\guDp\guDc\big)\big) e^{- H\big(\guDp\guDc\big)} \induD \\
 & \hspace*{6.5cm} \pi^z_{\Ln\setminus\t-u(\Delta)}\big(d\guDc\big)
\pizuD\big(d\guDp\big) \\
& = \int f(\tu(\g))\induD P_n(d\g).
\end{align*}
This DLR type equation is then verified for $\Pbarn^*$ by mixing 
\begin{equation*}
\int \fDK(\g)\indK \Pbarn^*(d\g) = \int f(\g) \indK \Pbarn^*(d\g).
\end{equation*}
Since we have the approximation 
\begin{equation*}
\left| \int f(\g) \indK \Pbarn^*(d\g) - \int f(\g) \indK \Pbarn(d\g) \right| \leq \normf \varepsilon,
\end{equation*}
we obtain finally (\ref{approx5}).
\paragraph{Step 6.} We show the last approximation 
\begin{equation}\label{approx6}
\left\vert\int f(\g) \indK \Pbarn(d\g) - \int f(\g) P(d\g) \right\vert \leq 2\normf\varepsilon.
\end{equation}
Using the local convergence of $(P_n)_{n\geq 1}$ to $P$ again, we have, for $n$ large enough
\begin{equation*}
\left\vert\int f(\g) \indK \Pbarn(d\g) - \int f(\g) \indK P(d\g) \right\vert \leq \normf\varepsilon.
\end{equation*}
With our choice of $K$ we have $P(\vert\gD\vert > K) \leq \varepsilon$, we obtain (\ref{approx6}).

\paragraph{Conclusion} Gathering approximations (\ref{approx1}), (\ref{approx2}), (\ref{approx3}), (\ref{approx4}), (\ref{approx5}) and (\ref{approx6}), we have finally 
\begin{equation*}
\left| \int \fD dP - \int f dP \right| \leq 16 \normf\varepsilon.
\end{equation*}
The inequality is true for every $\epsilon > 0$, this ends the proof of Theorem \ref{t.DLR}.

\section*{Acknowledgement}

 This work was supported in part by the Labex CEMPI (ANR-11-LABX-0007-01), the ANR project PPPP (ANR-16-CE40-0016) and by the CNRS GdR 3477 GeoSto.

\bibliographystyle{plain}
\bibliography{biblioV4}
\end{document}